\documentclass[11pt,a4paper]{amsart}
\usepackage{amscd,amssymb,amsopn,amsmath,amsthm,graphics,amsfonts,enumerate,verbatim,calc}
\usepackage[dvips]{graphicx}

\usepackage{amssymb,amsmath, color}

\headsep=1cm \numberwithin{equation}{section}
\newtheorem{theorem}{Theorem}[section]
\newtheorem{lemma}[theorem]{Lemma}
\newtheorem{notation}[theorem]{Notation}
\newtheorem{proposition}[theorem]{Proposition}
\newtheorem{corollary}[theorem]{Corollary}

\theoremstyle{definition}
\newtheorem{definition}[theorem]{Definition}
\theoremstyle{remark}
\newtheorem{remark}[theorem]{Remark}
\newtheorem{fact}[theorem]{Fact}
\newtheorem{example}[theorem]{Example}

\newtheorem{discussion}[theorem]{Discussion}

\newcommand{\grade}{\operatorname{grade}}

\newcommand{\pd}{\operatorname{p.dim}}

\newcommand{\Syz}{\operatorname{Syz}}

\newcommand{\Ext}{\operatorname{Ext}}
\newcommand{\Supp}{\operatorname{Supp}}
\newcommand{\Tor}{\operatorname{Tor}}

\newcommand{\Hom}{\operatorname{Hom}}

\newcommand{\depth}{\operatorname{depth}}

\newcommand{\fm}{\frak{m}}
\newcommand{\fp}{\frak{p}}

\newcommand{\fa}{\frak{a}}

\newcommand{\PP}{\mathbb{P}}

\begin{document}

\author[Asgharzadeh]{Mohsen Asgharzadeh }

\address{}
\email{mohsenasgharzadeh@gmail.com}

\title[A note on Hilbert-Kunz multiplicity ]
{A note on Hilbert-Kunz multiplicity }


\begin{abstract}
In this note we first give a new bound on $e_{HK}(\sim)$ the Hilbert-Kunz multiplicity of invariant rings, by the help of the Noether's bound.
Then, we simplify, extend  and present  applications of the reciprocity formulae due to L. Smith. He proved the formula over polynomial rings and his result is tight in the following sense:  Over complete intersection rings with isolated singularity
we show that the reciprocity formulae "$e_{HK}(R/I)+ e_{HK}(R/J) = e_{HK}(R/\underline{f})$"
is equivalent with  $\pd(I)<\infty$ when $I$ is an  $\fm$-primary unmixed ideal linked to $J$ along with  a regular sequence $\underline{f}$.
\end{abstract}

\maketitle

\smallskip

\section{ Introduction}

Throughout this note rings are of prime characteristic $p>0$.  Let $F^n(-)$ be the \textit{Peskine-Szpiro} functor.
The \textit{Hilbert-Kunz multiplicity }over $\ast$-local $(A,\fm)$  with an $\fm$-primary ideal $I$ is defined by $$e_{HK}(I,A):=\lim_{n\to\infty}\frac{\ell \left(F^n(A/ I)\right)}{p^{n\dim A}}.$$
This  introduced by  Kunz.  Monsky   proved such a limit exists \cite{mon}.
Hilbert-Kunz multiplicity encodes the singularity in prime characteristic \cite{wa}, et cetera. Some times computing  $e_{HK}(\sim)$ is
 difficult and finding bounds may be a reasonable task.  Also, $e_{HK}(\sim)$ behaves both similar to and different from those of Hilbert-Samuel multiplicity. In this note we present a new sample for this  behavior.

We use $e_{HK}(A)$ instead of $e_{HK}(\fm,A)$. We denote the usual Hilbert-Samuel multiplicity by $e(I)$. 
Suppose a finite group $G$ acts linearly over a polynomial ring with $n$ variables, denote the invariant ring by $A$ and  assume that the order of $G$ is prime to $p$. First we observe $e_{ HK}(A)\leq\frac{\binom{n-1+\mid G\mid}{n}}{|G|},$ as an easy application of Noether's bound  from invariant theory. Over 2-dimensional \textit{quotient singularity},  this has in the simple form $e_{ HK}(k[X,Y]^G)\leq\frac{\mid G\mid+1}{2}$. So $e_{}(k[X,Y]^G)\leq|G|+1$.
We show the bound is sharp in some cases but not in general. We do this in Section 3.

Section 4 motivated by the \textit{reciprocity formulae} of Larry Smith. This result connects the Hilbert-Kunz multiplicity of
linked ideals $I$ and $J$  via the complete intersection ideal $\frak a:=(\underline{f})$ over \textit{polynomial rings}, i.e., $$e_{HK}(I,R)+ e_{HK}(J,R) = e_{HK}(\frak a,R).$$ We first reprove Smith's result by a short argument. Then we extend  it from the polynomial rings to more general setting.
We give  corollaries of this result, see e.g. Corollary \ref{c}.

Suppose $X, Y\subset \PP^3$ are curves  \textit{directly
linked} by equations of degrees $d$ and $e$. It may be worth to  recall that Peskine and Szpiro  proved
the degrees and genus of $X$ and $Y$ are related by the following formulas:
\begin{enumerate}
\item[i)]$\deg X + \deg Y = de$
\item[ii)] $g(X) - g(Y) = (\deg X - \deg Y) (d + e - 4) /2.$
 \end{enumerate}
Up to our knowledge the above result regards as a root of liaison theory.
The following is our main result:
\begin{theorem}
Let $R$ be a local complete intersection with isolated singularity and of positive dimension.
Let $I$ be an unmixed primary ideal linked to $J$ via the complete intersection ideal $\frak a:=(\underline{f})$ of full length. The following are equivalent.
\begin{enumerate}
\item[i)] $e_{HK}(I,R)+ e_{HK}(J,R) = e_{HK}(\frak a,R).$
\item[ii)]  $\pd(I)<\infty$.
 \end{enumerate}
\end{theorem}
To prove this result, and among other things,  we use the \textit{corner powers} of unmixed ideals
developed by Vraciu,  and we apply the theory of \textit{generalized Hilbert-Kunz multiplicity}
that introduced by Dao and Smirnov. In particular, the generalized Hilbert-Kunz multiplicity has a new
application to the classical Hilbert-Kunz multiplicity.

\section{ preliminary}
All rings in this note are commutative and Noetherian.
Let $R$ be a ring of prime characteristic $p$. The assignment $a\mapsto a^p$ defines a ring homomorphism $F:R\to R$.  By $F^n(R)$, we mean $R$ as a group equipped with  left and right scalar multiplication from $R$
given by
$$a.r\star b = ab^{p^n}r, \  \ where \ \ a,b\in R  \ \ and \  \ r\in F^n(R).$$
Now we recall the Peskine-Szpiro functor. For a finitely generated $R$-module $M$, set $$F^n(M):=M\otimes F^n(R).$$
It is easy  to see $F^n(R/ \frak a)=R/ \frak a^{[p^n]}$. Let $\fa$ be an ideal $\frak a$ with a generating set
$\underline{x}:=x_{1}, \ldots, x_{r}$. The $i^{th}$  local cohomology of $M$ with
respect to $\fa$ is defined by
$$H^{i}_{\fa}(M):=\underset{n}{\varinjlim}\Ext^{i}_{R} (R/\fa^{n},
M).$$ One may calculate it as the $i^{th}$ cohomology of
\textit{$\check{C}ech$} complex of  a module $M$ with respect  to $\underline{x}$.
The grade of
$\frak a$ on  $M$ is defined by $$\grade_{R}(\frak a,M):=\inf\{i\in
\mathbb{N}_0|H_{\frak a}^{i}(M)\neq0\}.$$We use $\depth (M)$, when we deal with the maximal ideal of $\ast$-local rings.
Our next auxiliary tool is  \begin{enumerate}
\item[i)]$f_{gHK}^{M}(n):=\ell \left(H^0_{\fm}(F^n(M))\right)$
\item[ii)]  $e_{gHK}(M):=\lim_{n\to\infty}\frac{f_{gHK}^{M}(n)}{p^{n\dim R}},$
 \end{enumerate}
if the limit exists  and put $\infty$ otherwise.  Following Dao and Smirnov \cite{dao}, we call them the \textit{generalized Hilbert-Kunz function} and the generalized Hilbert-Kunz multiplicity. If $\ell(M)<\infty$ we are in the situation of \cite{mon} and we use $f_{HK}^M$ and $e_{HK}(M)$, respectively. If there is no doubt of confusion, we use  $e_{HK}(R/\fa)$ instead of $e_{HK}(\frak a,R)$. We need the following two facts:

\begin{fact}\label{locali}(See \cite[Proposition 8.2.5]{BH})
 Frobenius map commutes with the localization, i.e., $R_S\otimes_RF^n(M)\simeq F^n(R_S \otimes_RM)$
for any multiplicative closed subset $S\subset R$.
\end{fact}

\begin{fact}\label{kunz}
Over regular  rings the Frobenius map is flat, see \cite[Corollary 8.2.8]{BH}.
\end{fact}

\section{A bound on $e_{HK}(\sim)$  of invariant rings}
Throughout this section $R:=\bigoplus_{n\geq0}R_n$ is a standard graded algebra of a field  $R_0:=k$ of prime characteristic $p>0$ and $\fm:= \bigoplus_{n>0}R_n$  is the irrelevant ideal.

\begin{notation}
Let $G$ be a finite group  such that  $|G|$ is invertible in the field $k$, i.e., $(p,|G|)=1$.
\end{notation}

\begin{discussion}
Let $A$ be a ring and $\rho:G\to \verb"Aut"(A)$ be a representation of finite group $G$ by $\verb"Aut"(A)$ and  $|G|$ is invertible in $A$.
 Recall that  $\frak a\subset A$ is said to be \textit{$G$-stable}, if  $g(a)\in \frak a$ for all $g \in G$ and $a \in \frak a$. For any subset $X\subset A$
 by $XA$, we mean the ideal generated by $X$.
\end{discussion}
Let us recall the following  version of \textit{Noether's bound} due to D. Benson:

\begin{lemma}\label{no}(See \cite[Lemma 2.3.1]{smithb})
Adopt the above notation and let $\frak a$ be $G$-stable.
Then $\frak a^{\mid G\mid}\subseteq (\frak a^G)A$.
\end{lemma}

Suppose $G$ acts   on $R$ by degree-preserving  $k$-algebra homomorphisms.
This means that $g(R_n) \subseteq R_n$ for all $g \in G$ and $n \in \mathbb{N}$. Then $$A:=R^G = \{f \in R : g(f)=f \quad \forall g \in G\}$$ is the ring of invariants.  Denote
$\bigoplus_{i>0} A_i$ by $\frak n$.
Recall from \cite{wa} that:

\begin{fact}\label{f}
$e_{ HK}(A)=\frac{\dim (R/ \frak n R)}{|G|}.$
\end{fact}

\begin{proposition}\label{b}
Adopt the above notation. Then $e_{ HK}(A)\leq\frac{\binom{n-1+\mid G\mid}{n}}{|G|}.$
\end{proposition}

\begin{proof} As $G$ acts   on $R$ by degree-preserving  $k$-algebra homomorphisms, $\fm$ is $G$-stable. In view of Lemma \ref{no}, $ \fm^{\mid G\mid}\subseteq (\frak m^G)R$. Also, $(\frak m^G)R\subseteq \frak n R.$ Thus, $$\fm^{\mid G\mid}\subseteq (\frak m^G)R\subseteq \frak n R.$$
This induces  a surjection $$R/ \fm^{\mid G\mid}\longrightarrow R/\frak n R\longrightarrow 0,$$and so \[\begin{array}{ll}
\dim(R/ \frak n R)& \leq \dim(R/ \fm^{\mid G\mid})\\
&=\sum _{i=0}^{\mid G\mid-1} \dim (R_i)\\
&=\sum _{i=0}^{\mid G\mid-1}
\binom{n+i-1}{n-1}\\
&=\binom{n-1+\mid G\mid}{n}.
\end{array}\]Therefore, in view of Fact \ref{f},  $$e_{ HK}(A)\leq\frac{\binom{n-1+\mid G\mid}{n}}{|G|}.$$\end{proof}

\begin{corollary}
Over 2-dimensional invariant rings, one has: \begin{enumerate}
\item[i)] $e_{ HK}(k[X,Y]^G)\leq\frac{\mid G\mid+1}{2}$
\item[ii)]  $e_{}(k[X,Y]^G)\leq|G|+1$.
 \end{enumerate}
\end{corollary}

\begin{proof}
Set $n=2$ in Proposition \ref{b}. Part i) is now immediate. For part ii) we recall a result of Monsky \cite{mon}, i.e., over any ring $A$ one observes $$e_{}(A)/ (\dim A)^!\leq e_{ HK}(A).$$
Combining this along with i) we get the claim.
\end{proof}

\begin{corollary}
Suppose $|G|=2$. One has \begin{enumerate}
\item[i)] $e_{ HK}(R^G)\leq \frac{\dim R+1}{2}$.
\item[ii)]  $e_{}(R^G)\leq\frac{(\dim R+1)^!}{2}$.
 \end{enumerate}
\end{corollary}

Let us show these bounds are sharp.

\begin{example}
Suppose $|G|=2$ and $n=2$. Denote the corresponding  invariant ring by $A$. We assume that $A$ is not regular and such a thing happens. Then we get $$\frac{e(A)}{(\dim A)^!}\leq e_{ HK}(A)\leq \frac{3}{2}\quad (\ast)$$
As the ring $A$ is not regular, $e_{ HK}(A)>1$. This along with $(\ast)$ yields:\begin{enumerate}\item[i)] $e_{ HK}(A)=3/2$;
\item[ii)]  $e_{}(A)=3$.
 \end{enumerate}
\end{example}

\section{ Smith's reciprocity formulae}

Recall that  $I$ is\textit{ linked }to $J$ if there is a \textit{regular sequence} $\underline{f}:=f_1,\ldots,f_i$ in $I\cap J$ such that $I= (\underline{f} : J)$ and
 $J= (\underline{f} : I)$.  Smith proved a reciprocity formulae over polynomial rings \cite[Theorem 5.7]{smith}. First, we extend his result to Gorenstein rings by a new and short argument.

\begin{proposition}\label{s}
Let $R$ be a d-dimensional   Gorenstein  ring, $I$ and $J$  primary to the maximal  ideal $\fm$ linked via  a regular sequence $\underline{f}:=f_1,\ldots,f_d$. Then $$\ell(R/I)+\ell(R/J)=\ell(R/\underline{f}).$$
\end{proposition}

\begin{proof}Note that
\[\begin{array}{ll}
 \ell(R/I)+\ell(R/J)&\stackrel{1}=\ell(H^0_{\fm}(R/I))+\ell( R/J )\\
&\stackrel{2}= \ell( \Ext^{d}_{R}(R/I,R))+\ell( R/J )\\
&\stackrel{3}=\ell( \Ext^{0}_{R}(R/I,R/\underline{f}))+\ell( R/J )   \\
&\stackrel{4}=\ell(J/\underline{f})+\ell( R/J ) \\
&\stackrel{5}= \ell(R/\underline{f}),
\end{array}\]where:
\begin{enumerate}
\item[1)] Since $R/I$ is of finite length.
\item[2)]  This is the  local duality.
\item[3)] Keep in mind that $\underline{f}$ is a regular sequence
in the Jacobson radical of the ring. An easy induction implies that$$\Ext^{0}_{R}(R/I,R/\underline{f})\simeq\Ext^{d}_{R}(R/I,R).$$
 \item[4)] This is in the definition of linked ideals.
  \item[5)] This follows by the following short exact sequence: $$0\longrightarrow J/\underline{f}\longrightarrow R/\underline{f} \longrightarrow R/J \longrightarrow 0.$$
 \end{enumerate}The proof is now complete.
 \end{proof}

\begin{corollary}\label{c}
Let $(R,\fm)$ be a Gorenstein local ring  of dimension zero with an ideal $I$ in which $(0:I)=I$. Then $\ell(R)\in 2\mathbb{N}$.
\end{corollary}

\begin{proof}
 As $(0:I)=I$, we get that \[\begin{array}{ll}
 2\ell(R/I)&=\ell(H^0_{\fm}(R/I))+\ell( R/(0:I) )\\
&=\ell( \Hom_{R}(R/I,R))+\ell( R/(0:I) )   \\
&=\ell((0:I))+\ell( R/(0:I) ) \\
&= \ell(R).
\end{array}\]So, $\ell(R)\in 2\mathbb{N}$.
\end{proof}

\begin{example}(i) Here we remark that the length of a zero-dimensional Gorenstein local ring can be every thing, e.g., $K[X]/(X^n)$.
To see  examples of high embedding dimension, look at the zero-dimensional Gorenstein ring$$R:=K[X_1,\ldots,X_n]/(  X_1^2-X_2^2,\ldots,  X_1^2-X_n^2,X_1X_2,X_1X_3,\ldots,X_{n-1}X_n).$$
So, $\{1,X_1,\ldots,X_n,X_1^2\}$ is a  base for $R$ and $\{X_1,\ldots,X_n\}$ is a  base for $\fm / \fm^2$.

(ii) Let $(R,\fm, K)$ be a Gorenstein local ring containing a field with $\fm^2=0$. By this property, $(0:\fm)=\fm$. In view of the above Corollary, $\ell(R)\in 2\mathbb{N}$. Let us check this by hand.
As $R$ is Gorenstein, its socle is principal. Therefore, $\fm$ is principal. Denote its generator by $r$. The assignment $r\mapsto X$ yields $R\simeq K[X]/(X^2)$. In particular,
$\ell(R)=2$ as claimed.
\end{example}

\begin{remark}(See \cite[Example 6.6]{ku}) The concept of \textit{numerically Roberts rings} introduced in \cite{ku}.
Suppose $R$ is  a homomorphic image of a regular local ring and let
$d = \dim R$. The following are numerically Roberts rings:
\begin{enumerate}
\item[i)] $d<2$
\item[ii)] $d=2$ and $R$ is  equi-dimensional.
\item[iii)]$d=3$ and $R$ is Gorenstien.
 \end{enumerate}
\end{remark}
The following  extends \cite[Corollary 5.8]{smith} via a new argument.

\begin{corollary} \label{fpd}
Let $R$ be a $d$-dimensional Gorenstein ring, $I$ and $J$  be two ideals  of finite projective dimension linked through  a regular sequence $\underline{f}:=f_1,\ldots,f_d$. Assume one of the following holds: \begin{enumerate}
\item[i)] Suppose in addition, $R$  is standard graded over a field, $I$ and $J$ are homogenous.
\item[ii)]  Suppose in addition, $R$ is a local complete intersection ring.
 \item[iii)] Suppose in addition, $R$ is  local and $\dim R<4$.
 \end{enumerate}
Then $e_{HK}(I,R)+ e_{HK}(J,R) = e_{HK}(\underline{f},R).$
\end{corollary}

\begin{proof}i) Recall that over a graded ring of dimension $d$ with a graded module $M$  of finite length and of finite projective dimension, a celebrated result of Peskine and Szpiro says:
$$\ell(F^n(M))=p^{nd}\ell(M).\footnote{There are 3-dimensional Cohen-Macaulay rings of characteristic $p$ such that $\ell(F^n(M))\neq p^{3n}\ell(M)$, for a suitable module of finite length, by \cite{Rob}.}$$ In particular, $e_{HK}(I,R)=\ell(R/I)$. Now, Proposition \ref{s} proves the desired claim.

ii) First we note that $\pd(R/\underline{f})<\infty$.   Any complete intersection ring is a numerically Roberts ring. In view of \cite[Theorem 6.4]{ku}, the
Hilbert-Kunz multiplicity of an ideal primary to the maximal ideal and of finite projective
dimension is  its colength. Now, Proposition \ref{s} proves the desired claim.

iii) In view of the above Remark this follows in a similar vein as ii).
\end{proof}

Recall that a module $M$ is called \textit{generalized Cohen-Macaulay }if  $H^i_{\fm}(M)$ is finitely generated  for all $i<\dim M$.
Recall from \cite{dao}:

\begin{definition}
Let $M$ be generalized Cohen-Macaulay and $i<\dim M$. The $i$-th generalized Hilbert-Kunz multiplicity of $M$ is
$$e_{gHK}^i(M):=\lim_{n\to\infty}\frac{\ell \left(H^i_{\fm}(F^n(M))\right)}{p^{n\dim R}},$$
if it exists  and put $\infty$ otherwise.
\end{definition}

\begin{remark}
Let $R$ be a  regular ring and $M$ a module of finite length. Then $F^n(M)$ is of finite length. Indeed, we do  induction by $\ell:=\ell(M)$. If $\ell=1$, then $M=R/ \fm$ and the claim  is clear in this case.
Now suppose, inductively, that $\ell > 1$ and the result has been proved for
modules of length less than $\ell$. Look at the exact sequence $$0\longrightarrow N\longrightarrow M\longrightarrow R/ \fm \longrightarrow 0$$
where $\ell(N)=\ell-1$. By inductive hypothesis $F^n(N)$ is of finite length. By Fact \ref{kunz}, there is an exact sequence$$0\longrightarrow F^n(N)\longrightarrow F^n(M)\longrightarrow F^n(R/ \fm) \longrightarrow 0.$$The claim is now clear.
\end{remark}

\begin{proposition}\label{sc}
Let $R$ be a  regular ring of dimension $d$, $I$ and $J$ be two ideals linked via  a regular sequence $\underline{f}:=f_1, \ldots, f_g$ for some $g<d$. Suppose
in addition that $R/I$ is generalized Cohen-Macaulay of dimension $m$. Then  for all $1\leq i \leq m-1$ one has $$e_{gHK}^i(R/I)= e_{gHK}^{m-i}(R/J).$$
\end{proposition}

\begin{proof} Set $\underline{f}^{p^n}:=f_1^{p^n}, \ldots, f_g^{p^n}$. As $R$ is regular, the Frobenius map is flat. Hence \begin{enumerate}
\item[1)]$ I^{[p^n]}= (\underline{f}^{p^n} : J^{[p^n]})$,
\item[2)] $J^{[p^n]}= (\underline{f}^{p^n} : I^{[p^n]})$ and
\item[3)] $\underline{f}^{p^n}$ is a regular sequence.
 \end{enumerate}Thus $I^{[p^n]}$ and  $J^{[p^n]}$ are linked through $\underline{f}^{p^n}$.
As $\Supp(F^n(M))=\Supp(M))$, one has $\dim(F^n(M))=\dim (M)$. In particular, $\dim R/I^{[p^n]}=m$. Also, exact functors commute with the local cohomology modules.
By the above Remark, Fact \ref{kunz} and in view of
$$F^n(H^i_{\fm}(R/I))\simeq H^i_{\fm}(F^n(R/I))\simeq
H^i_{\fm}(R/I^{[p^n]}),$$ we observe that $H^i_{\fm}(R/I^{[p^n]})$ is finitely generated  for all $i<m$.
Let $1\leq i \leq m-1$. In view of \cite[Corollary 3.3]{sch}, $$H^i_{\fm}(R/I^{[p^n]})\simeq\Hom(H^{m-i}_{\fm}(R/J^{[p^n]}),E).$$So, $$\ell(H^i_{\fm}(R/I^{[p^n]}))=\ell(H^{m-i}_{\fm}(R/J^{[p^n]})).$$
This yields $e_{gHK}^i(R/I)= e_{gHK}^{m-i}(R/J)$ as claimed.
\end{proof}

We cite the next result for the convenience of the reader.

\begin{lemma}\label{av} (Avramov-Miller, Dutta) Let R be a local complete intersection and M a
finitely generated R-module. Then the vanishing of $\Tor_R^i(M, F^n(R))$ for one value each of $i > 0$
and $n > 0$ implies that M has finite projective dimension.
\end{lemma}

Let $(R, m)$ be a characteristic $p$ Gorenstein ring. If $I \lhd R$ is an unmixed ideal of height $g$,
let $\frak a \subset I$ be a Gorenstein ideal of finite projective dimension and of
height $g$ (e.g., a parameter ideal), and let $J := (\frak a : I)$. Now, by the help of \cite[Proposition 1]{V}, the following is well-define:

\begin{notation}
Following \cite{V}, define
$I^{<q>} := (\frak a^{[q]}: J^{[q]})$ for all $q=p^n$.
\end{notation}

\begin{lemma}\label{newpower}
Let $(R, \fm)$ be a  ring of prime characteristic $p$,
let $I$ be an unmixed ideal and $q = p^n$. The following holds:
\begin{enumerate}
\item[i)] If $R$ is Gorenstein, $I^{[q]} \subseteq I^{<q>} \ \ \forall q$ and the equality holds when $\pd(I)<\infty$.
\item[ii)] Suppose $R$ is complete intersection. Then $I^{[q]}=I^{<q>}  \forall q\Leftrightarrow\pd(I)<\infty$.
\item[iii)] Suppose $R$ is complete intersection. Then $I^{[q]}=I^{<q>}  \exists q\Leftrightarrow\pd(I)<\infty$.
\end{enumerate}
\end{lemma}

\begin{proof}
i): This is in \cite[Proposition 2]{V}.

ii): First note that complete intersection rings are Gorenstein. The desired claim 
 is in \cite[Proposition 2]{V} when the ring is hypersurface.
Here, we sketch that proof. Let $N$ be the cokernel of the natural inclusion $R/I\hookrightarrow \oplus_n R/ \frak a$
defined via the matrix $(f_1,\ldots,f_n)$. Vraciu's main trick is as follows:$$I^{<q>}/I^{[q]} =\Tor^R_1(N,F^n(R)) \quad(\ast),$$
see the proof of \cite[Proposition 2]{V}.  Now, suppose $I^{<q>}=I^{[q]}$. Apply this throughout the complete intersection property to conclude that $\Tor^R_1(N,F^n(R))=0$. So, Lemma \ref{av} implies that
$\pd_R(N)<\infty$. By definition of $N$, we observe $\pd_R(R/I)<\infty$.

iii) This is similar to ii).
\end{proof}

In what follows we will use the following two results on the concept of generalized Hilbert-Kunz functions.
\begin{lemma}\label{asstor}
(See \cite[Corollary 4.7]{dao}) Let $R$ be a local complete intersection with isolated singularity. The following
are equivalent:
\begin{enumerate}
\item[i)] $e_{gHK}(M) = 0$,
\item[ii)] $f^M_{gHK}(n) = 0$ for all n,
\item[iii)] $\pd_R M < \dim R$.\end{enumerate}
\end{lemma}

\begin{lemma}\label{tor1}
(See \cite[Lemma 3.10]{dao}) Let $R$ be a formally equi-dimensional local ring of positive depth,
then $H^0_{\fm}\left(\Tor^R_1( M,F^n(R))\right)=H^0_{\fm}(\Syz(M))$.
\end{lemma}

\begin{corollary}\label{tor}
Suppose in addition to Lemma \ref{tor1}, $R$ has isolated singularity. Then $$\Tor^R_1( M,F^n(R))=H^0_{\fm}(\Syz(M)).$$
\end{corollary}

\begin{proof}
It is enough to show $\Tor^R_1( M,F^n(R))$ is of finite length. Equivalently, $$\Tor^R_1( M,F^n(R))_{\fp}=0 \quad\forall\fp\neq \fm.$$
Let $\fp$ be a non-maximal prime ideal. Due to Fact \ref{kunz} and the isolated singularity, $$\Tor^{R_{\fp}}_1(\sim,F^n(R_{\fp}))=0.$$
By Fact \ref{locali},$$\Tor^R_1( M,F^n(R))_{\fp}\simeq\Tor^{R_{\fp}}_1( M_{\fp},F^n(R_{\fp}))=0,$$
and this completes the argument.\end{proof}

Now we are ready to prove the following:

\begin{theorem}\label{main}
Let $R$ be a local complete intersection with isolated singularity and of positive dimension.
Let $I$ be an unmixed ideal linked to $J$ throughout a regular sequence $\underline{f}$ of full length. The following are equivalent:
\begin{enumerate}
\item[i)] $e_{HK}(I,R)+ e_{HK}(J,R) = e_{HK}(\underline{f},R),$
\item[ii)]  $\pd(I)<\infty$.
 \end{enumerate}
\end{theorem}

\begin{proof}
$i)\Rightarrow ii)$: First we remark that any Cohen-Macaulay ring is formally equidimensional. Suppose on the contrary that
$\pd(I)=\infty$.  Let $d:=\dim R$. Let $N$ be the cokernel of the natural inclusion $R/I\hookrightarrow \oplus_d R/ \frak a$
defined via the matrix $(f_1,\ldots,f_d)$. Then by Lemma \ref{newpower},
$$0\neq  I^{<q>}/I^{[q]} =\Tor^R_1(N,F^n(R)).$$ On the other hand by Corollary \ref{tor}
$$\ell\left(\Tor^R_1(N,F^n(R))\right)=f_{gHK}^{\Syz( N)}(n).$$We note that $\pd(N)=\infty$ and the same thing holds for its syzygies.
By Lemma \ref{asstor}, $$\lim_{n\to\infty}\frac{\ell(I^{<q>}/I^{[q]})}{q^d} \neq 0.$$ In the light of
\cite[Proposition 3]{V} we  observe
$$\ell(R/I^{<q>})+\ell(R/J^{[q]})=\ell(R/\underline{f}^{[q]}).$$Putting this along with i)
we get that $$\lim _{n\to\infty}\frac{\ell(I^{<q>}/I^{[q]})}{q^d}= 0,$$a contradiction.

$ii)\Rightarrow i)$: This is in Corollary \ref{fpd}.
\end{proof}


\end{document}